\RequirePackage{fix-cm}
\let\oldvec\vec
\documentclass[smallextended]{svjour3}       
\let\vec\oldvec
\smartqed  
\usepackage{graphicx}
\usepackage{amssymb,amsmath}
\usepackage{bbm}
\usepackage{graphicx,comment}
\usepackage{url}
\usepackage{algorithm,algorithmic}
\usepackage{hhline}
\usepackage{multirow,color}

\newcommand{\R}{\mathbb{R}}
\newcommand{\gf}{\nabla f}
\newcommand{\C}{\mathcal{C}}

\newtheorem{prop}{Proposition}

\graphicspath{{img/}}
\begin{document}

\title{Stochastic Backward Euler: An Implicit Gradient Descent Algorithm for $k$-means Clustering 
}
\subtitle{}


\author{Penghang~Yin \and Minh~Pham \and
Adam~Oberman \and Stanley~Osher
}


\institute{Penghang~Yin \and Minh~Pham \and Stanley~Osher \at
              Department of Mathematics, University of California, Los Angeles, Los Angeles, CA 90095. \\
            \email{(yph, minhrose, sjo)@math.ucla.edu}           
           \and
           Adam Oberman \at
           Department of Mathematics and Statistics, McGill University, Montreal, Canada.\\
           \email{adam.oberman@mcgill.ca} 
 }

\date{Received: date / Accepted: date}

\titlerunning{Stochastic Backward Euler For $k$-means}%

\maketitle

\begin{abstract}
In this paper, we propose an implicit gradient descent algorithm for the classic $k$-means problem. The implicit gradient step or backward Euler is solved via stochastic fixed-point iteration, in which we randomly sample a mini-batch gradient in every iteration. It is the average of the fixed-point trajectory that is carried over to the next gradient step. We draw connections between the proposed stochastic backward Euler and the recent entropy stochastic gradient descent (Entropy-SGD) for improving the training of deep neural networks. Numerical experiments on various synthetic and real datasets show that the proposed algorithm provides better clustering results compared to $k$-means algorithms in the sense that it decreased the objective function (the cluster) and is much more robust to initialization. 

\end{abstract}

\section{Introduction}
The $k$-means method appeared in vector quantization in signal processing, and had now become popular for clustering analysis in data mining. In the seminal paper \cite{lloyd}, Lloyd proposed a two-step alternating algorithm  that quickly converges to a local minimum. Lloyd's algorithm is also known as an instance of the more general Expectation-Maximization (EM) algorithm applied to Gaussian mixtures. In \cite{Bottou_95}, Bottou and Bengio cast Lloyd's algorithm  as Newton's method, which explains its fast convergence.

Aiming to speed up Lloyd's algorithm, Elkan \cite{elkan} proposed to keep track of the distances between the computed centroids and data points, and then cleverly leverage the triangle inequality to eliminate unnecessary computations of the distances. Similar techniques can be found in \cite{yingyang}. It is worth noting that these algorithms do not improve the clustering quality of Lloyd's algorithm, but only achieve acceleration. However, there are well known examples where poor initialization can lead to low quality local minima for Lloyd's algorithm. Random initialization has been used to avoid these low quality fixed points.  The article \cite{kmeans++} introduced a smart initialization scheme such that the initial centroids are well-separated, which gives more robust clustering than random initialization. 

We are motivated by problems with very large data sets, where the cost of a single iteration of Lloyd's algorithm can be expensive.  Mini-batch \cite{mnbatch,stochastic_kmeans} was later introduced to adapt $k$-means for large scale data with high dimensions. The centroids are updated using a randomly selected mini-batch rather than all of the data. Mini-batch (stochastic) $k$-means has a flavor of stochastic gradient descent whose benefits are twofold. First, it dramatically reduces the per-iteration cost for updating the centroids and thus is able to handle big data efficiently. Second, similar to its successful application to deep learning \cite{deep_learning}, mini-batch gradient introduces noise in minimization and may help to bypass some bad local minima. Furthermore, the aforementioned Elkan's technique can be combined with mini-batch $k$-means for further acceleration \cite{nested_mnbatch}.

In this paper, we propose a backward Euler based algorithm for $k$-means clustering. Fixed-point iteration is performed to solve the implicit gradient step. As is done for stochastic mini-batch $k$-means, we compute the gradient only using a mini-batch of samples instead of the whole data, which enables us to handle massive data. Unlike the standard fixed-point iteration, the proposed stochastic fixed-point iteration outputs an average over its trajectory. Extensive experiments show that, with proper choice of step size for stochastic backward Euler, the proposed algorithm can improve over EM and Mini-batch EM and locate an improved minimum with decreased objective value. 

 In other words, while Lloyd's algorithm is effective with \emph{a full gradient oracle} we achieve better performance with the weaker \emph{mini-batch gradient oracle}.
We are motivated by  recent work by two of the authors \cite{deep_relax} which applied a similar algorithm to accelerate the training of Deep Neural Networks. 

\section{Stochastic backward Euler}
The celebrated proximal point algorithm (PPA) \cite{proximal_point} for minimizing some function $f(x)$ is:
\begin{equation}\label{prox}
x^{k+1} = \mathrm{prox}_{\gamma f} (x^k) := \arg\min_x f(x) + \frac{1}{2\gamma}\|x - x^k \|^2.
\end{equation}
PPA has the advantage of being monotonically decreasing, which is guaranteed for any step size $\gamma>0$. Indeed, by the definition of $x^{k+1}$ in (\ref{prox}), we have 
$$
f(x^{k+1})\leq f(x^k) - \frac{1}{2\gamma}\|x^{k+1} - x^k\|^2.
$$
When $\gamma\in[c,\frac{1}{L(\gf)})$ for any $c>0$ with $L(\gf)$ being the Lipschitz constant of $\gf$, the (subsequential) convergence to a stationary point is established in \cite{PPM}. If $f$ is differentiable at $x^{k+1}$, it is easy to check that the following optimality condition to (\ref{prox}) holds
$$
\gf(x^{k+1}) + \frac{1}{\gamma}(x^{k+1}-x^k) = 0.
$$
By rearranging the terms, we arrive at implicit gradient descent or the so-called backward Euler:
\begin{equation}\label{BE}
x^{k+1} = x^{k} - \gamma\nabla f(x^{k+1}).
\end{equation}
When $\gf$ has the Lipschitz constant $L(\gf)$ and $\gamma<\frac{1}{L(\gf)}$, the fixed point iteration 
\begin{equation}\label{fp}
y^{l+1} = x^k - \gamma \nabla f(y^l)
\end{equation}
is a viable option for updating $x^{k+1}$ by solving the equation.
$$
x = x^k - \gamma \nabla f(x).
$$
It is essentially the gradient descent 
$$
y^{l+1} = y^l - \tau \Big( \nabla f(y^l) + \frac{1}{\gamma}(y^l - x^k) \Big)
$$
on (\ref{prox}) by choosing the step size $\tau = \gamma$.

\begin{prop}
If $\gamma<\frac{1}{L(\nabla f)}$, then we have
\begin{itemize}
\item[(a)] $f(x) + \frac{1}{2\gamma}\|x - x^k \|^2$ is strongly convex, and the proximal problem (\ref{prox}) has a unique solution $y^*$.
\item[(b)] The fixed point iteration (\ref{fp})
generates a sequence $\{y^l\}$ converging to $y^*$ at least linearly.
\end{itemize}
\end{prop}
See \cite[Proposition 1.2.3]{Ber08}.

\bigskip

Let us consider $k$-means clustering for a set of data points $\{p_i\}_{i=1}^N$ in $\R^d$ with $K$ centroids $\{x_j\}_{j=1}^K$. Denoting $x = [x_1, \dots, x_K]^{\top}\in\R^{Kd}$, we seek to minimize
\begin{equation}\label{kmeans}
\min_{x\in\R^{Kd}} \; \phi(x): = \frac{1}{2N} \sum_{i=1}^N  \min_{1\leq j\leq K} \| x_j - p_i \|^2.
\end{equation}
Note that $\phi$ is non-differentiable at $x$ if there exist $p_i$ and $j_1 \neq j_2$ such that 
\begin{equation*}\label{non-diff}
j_1, j_2 \in \arg\min_{1\leq j\leq K} \|x_{j}-p_i\|^2.
\end{equation*}
This means that there is a data point $p_i$ which has two or more distinct nearest centroids $x_{j_1}$ and $x_{j_2}$. The same situation may happen in the assignment step of Lloyd's algorithm. 
In this case, we simply assign $p_i$ to one of the nearest centroids. With that said, $\phi$ is basically piecewise differentiable. By abuse of notation, we can define the 'gradient' of $\phi$ at any point $x$ by
\begin{equation}\label{hessian}
\nabla \phi(x) = \frac{1}{N}[\sum_{i\in\C_1} (x_1-p_i), \dots, \sum_{i\in\C_K} (x_K-p_i)]^{\top},
\end{equation}
where $\C_j$ denotes the index set of the points that are assigned to the centroid $x_j$. From now on and for the rest of the paper, we denote the piecewise gradient by $\nabla \phi$ as stated in (\ref{non-diff}), and none of the results depends on the specific assignment of ambiguous data points $p_i$. Similarly, we can compute the 'Hessian' of $\phi$ as was done in \cite{Bottou_95}:
$$
\nabla^2 \phi(x) = \frac{1}{N}\mathrm{Diag}\big(|\C_1|\mathbf{1}_{(|\C_1|)}, \dots, |\C_K|\mathbf{1}_{(|\C_K|)}\big),
$$
where $\mathbf{1}_{(n)}$ is an $n$-D vector of all ones. In what follows, we analyze how the fixed point iteration (\ref{fp}) works on the piecewise differentiable $\phi$ with discontinuous $\nabla \phi$. 

\begin{definition}
$g$ is piecewise Lipschitz continuous on $\Omega$ with Lipschitz constant $L$, if $\Omega$ can be partitioned into a finite number of sub-domains $\Omega_I$ satisfying $\cup_{I} \Omega_I = \mathbb{R}^{Kd}$, $\Omega_I \cap \Omega_J = \emptyset $, $\forall \, I\neq J$, 
and $g$ is Lipschitz continuous in each sub-domain $\Omega_I$, i.e., for each $\Omega_I$ we have
$$\|g(x) - g(y)\| \le L \|x-y\| \quad \forall x,y \in \Omega_I$$
\end{definition}
We can see that $\nabla \phi$ is at most piecewise $\frac{1}{K}$-Lipschitz continuous. The following result proves the convergence of fixed point iteration on $k$-means problem.

\begin{theorem}\label{thm:fp}
Let $\phi$ be the $k$-means objective function defined in (\ref{kmeans}). Suppose $\nabla \phi$ is piecewise $L$-Lipschitz. If $\gamma < 1/L$, then the fixed point iteration for minimizing $h(x):=\phi(x) +\frac{1}{2\gamma}\|x-x^k\|^2$ given by
$$
y^{l+1} = x^k - \gamma \nabla \phi(y^l)
$$
with the initialization $y^0 = x^k$ satisfies
\begin{itemize}
\item[(a)] $h(y^{l+1})\leq h(y^l) - (\frac{1}{2\gamma} - \frac{L}{2})\|y^{l+1} - y^l\|^2$ and $\|y^{l+1} - y^{l}\|\to 0$ as $l\to \infty$.
\item[(b)] $\{y^l\}$ is bounded. Moreover, if any limit point $y^*$ of a convergent subsequence of $\{y^l\}$ lies in the interior of some sub-domain, then the whole sequence $\{y^l\}$ converges to $y^*$ with a locally linear rate, which is a fixed point obeying 
$$
y^* = x^k -\gamma\nabla \phi(y^*).
$$
\end{itemize}
\end{theorem}

\begin{proof}
(a) We know that $\phi$ is piecewise quadratic. Suppose $y^l\in\Omega_I$ (note that $y^l$ could be on the boundary), then $\phi$ has a uniform expression restricted on $\Omega_I$ which is a quadratic function, denoted by $\phi_{\Omega_I}$. We can extend the domain of $\phi_{\Omega_I}$ from $\Omega_I$ to the whole $\R^{Kd}$, and we denote the extended function still by $\phi_{\Omega_I}$. Since $\phi_{\Omega_I}$ is quadratic, $\nabla\phi_{\Omega_I}$ is $L$-Lipschitz continuous on $\R^{Kd}$. Then we have the following well-known inequality
\begin{align*}
\phi_{\Omega_I}(y^{l+1})& \leq  \phi_{\Omega_I}(y^l) + \langle \nabla\phi_{\Omega_I}(y^l),y^{l+1} - y^l \rangle + \frac{L}{2}\|y^{l+1}-y^l\|^2 \\
&  = \phi(y^l) + \langle \nabla\phi(y^l),y^{l+1} - y^l \rangle + \frac{L}{2}\|y^{l+1}-y^l\|^2.
\end{align*}
Using the above inequality and the definition of $\phi$, we have
\begin{align*}
h(y^{l+1}) & = \phi(y^{l+1}) + \frac{1}{2\gamma}\|y^{l+1}-x^k\|^2 \leq \phi_{\Omega_I}(y^{l+1}) + \frac{1}{2\gamma}\|y^{l+1}-x^k\|^2\\
& \leq \phi(y^l) + \langle \nabla\phi(y^l),y^{l+1} - y^l \rangle + \frac{L}{2}\|y^{l+1}-y^l\|^2 + \frac{1}{2\gamma}\|y^{l+1}-x^k\|^2 \\
& = \phi(y^l) + \langle \nabla\phi(y^l),y^{l+1} - y^l \rangle + (\frac{L}{2}-\frac{1}{2\gamma})\|y^{l+1}-y^l\|^2 \\
& \qquad \qquad \qquad \qquad \qquad + \frac{1}{2\gamma}\|y^l - x^k\|^2 + \frac{1}{\gamma}\langle y^{l+1}-x^k, y^{l+1} - y^{l}\rangle \\
& = h(y^l) -(\frac{1}{2\gamma}-\frac{L}{2})\|y^{l+1}-y^l\|^2 + \langle\frac{1}{\gamma}(y^{l+1} -x^k)+\nabla \phi(y^l), y^{l+1} -y^l \rangle \\
& = h(y^l) -(\frac{1}{2\gamma}-\frac{L}{2})\|y^{l+1}-y^l\|^2.
\end{align*}
In the second equality above, we used the identity
$$
\frac{1}{2}\|a-b\|^2 + \langle a,b \rangle = \frac{1}{2}\|a\|^2 + \frac{1}{2}\|b\|^2
$$
with $a = y^{l+1}-y^l$ and $b=y^{l+1} - x^k$. Since $\gamma<\frac{1}{L}$, $\{h(y^l)\}$ is monotonically decreasing. Moreover, since $h$ is bounded from below by 0, $\{h(y^l)\}$ converges and thus $\|y^{l+1}-y^l\|\to 0$ as $l\to\infty$.

\medskip

(b) Since $h(y)\to \infty$ as $y\to\infty$, combining with the fact that $h(y^l)\leq h(y^{l+1})$, we have $\{y^l\}\subseteq \{y\in\R^{Kd}: h(y)\leq h(y^0)\}$ is bounded. Consider a convergent subsequence $\{y^{l_m}\}$ whose limit $y^*$ lies in the interior of some sub-domain. Then for sufficiently large $l_m$, $\{y^{l_m}\}$ will always remain in the same sub-domain in which $y^*$ lies and thus $\lim_{l_m\to \infty}\nabla\phi(y^{l_m}) = \nabla\phi(y^*)$. Since by (a), $\|y^{l+1}-y^{l}\|\to 0$, we have $\|\nabla\phi(y^{l+1})-\nabla\phi(y^{l})\| = \frac{1}{\gamma}\|y^{l}-y^{l-1}\|\to 0$ as $l\to\infty$. Therefore,
\begin{align*}
0 & = \lim_{l_m\to\infty} y^{l_m} - x^k + \gamma \nabla\phi(y^{l_m-1}) = \lim_{l_m\to\infty} y^{l_m} - x^k + \gamma \nabla\phi(y^{l_m}) \\
& = y^* - x^k + \gamma\phi\nabla(y^*),
\end{align*}
which implies $y^*$ is a fixed point. Furthermore, by the piecewise Lipschitz condition,
$$
\|y^{l_m+1}-y^*\| = \gamma\|\nabla \phi(y^{l_m}) - \nabla \phi(y^*)\|\leq L\gamma\|y^{l_m}-y^*\|.
$$
Since $L\gamma<1$, when $l_m$ is sufficiently large, $y^{l_m+1}$ is also in the same sub-domain containing $y^*$. By repeatedly applying the above inequality for $l > l_m$, we conclude that $\{y^l\}$ converges to $y^*$.

\end{proof}

\begin{remark}
This result can be extended to objective functions that are the pointwise infimum of a set of a finite number of  Lipschitz differentiable functions.
\end{remark}
\subsection{Algorithm description}

Instead of using the full gradient $\nabla\phi$ in fixed-point iteration, we adopt a randomly sampled mini-batch gradient 
$$
\nabla_l \phi = \frac{1}{M}[\sum_{i\in\C_1^l} (x_1-p_i), \dots, \sum_{i\in\C_K^l} (x_K-p_i)]^{\top}
$$ 
at the $l$-th inner iteration. Here, $\C_j^l$ denotes the index set of the points in the $l$-th mini-batch associated with the centroid $x_j$ obeying $\sum_{j=1}^K |\C_j^l| = M$. The fixed-point iteration outputs a forward looking average over its trajectory. Intuitively averaging greatly stabilizes the noisy mini-batch gradients and thus smooths the descent. We summarize the proposed algorithm in Algorithm \ref{alg}. Another key ingredient of our algorithm is an aggressive initial step size $\gamma^0\approx K$, which helps bypass bad local minimum at the early stage. Unlike in deterministic backward Euler, diminishing step size is needed to ensure convergence. But $\gamma$ should decay slowly because large step size is good for a global search.

\begin{algorithm}
\caption{Stochastic backward Euler for $k$-means.}
\label{alg}
\textbf{Input}: number of clusters $K$, step size $\gamma^0\approx K$, mini-batch size $M$, averaging parameter $\alpha>0$, step size decay parameter $\beta\lessapprox 1$.\\
\textbf{Initialize}: centroid $x^0$.
\begin{algorithmic}
\FOR {$k = 1,\dots,\mathrm{omaxit}$}
\STATE $y^{0,k} = x^{k-1}$
\STATE $x^k = y^{0, k}$
\FOR {$l = 1,\dots, \mathrm{imaxit}$}
\STATE Randomly sample a mini-batch gradient $\nabla_l \phi$. 
\STATE $y^{l,k} = x^{k-1} - \gamma^k\nabla_l \phi(y^{l-1, k})$
\STATE $x^{k} = \alpha x^{k} + (1-\alpha)y^{l, k}$
\ENDFOR
\STATE $\gamma^{k} = \beta\gamma^{k-1}$
\ENDFOR
\end{algorithmic}
\textbf{Output}: $x^{\mathrm{omaxit}}$
\end{algorithm}

%

\subsection{Related work}
Chaudhari et al. \cite{esgd} recently proposed the entropy stochastic gradient descent (Entropy-SGD) algorithm to tackle the training of deep neural networks. Relaxation techniques arising in statistical physics were used to change the energy landscape of the original non-convex objective function $f(x)$ yet with the minimizers being preserved, which allows easier minimization to obtain a 'good' minimizer with a better geometry. More precisely, they suggest to replace $f(x)$ with a modified objective function $f_{\gamma}(x)$ called local entropy \cite{local_entropy} as follows
$$f_{\gamma}(x): = -\frac{1}{\beta} \log \big( G_{\beta^{-1}\gamma} * \exp ( - \beta f(x) ) \big),
$$
where $G_{\gamma} (x) = (2\pi \gamma)^{-d/2} \exp \big( - \frac{|x|^2}{2\gamma} \big)$ is the heat kernel. The connection between Entropy-SGD and nonlinear partial differential equations (PDEs) was later established in \cite{deep_relax}. The local entropy function $f_{\gamma}$ turns out to be the solution to the following viscous Hamilton-Jacobi (HJ) PDE
at $t=\gamma$
\begin{equation} \label{eqn1}
u_t = -\frac{1}{2} |\nabla u|^2 + \frac{\beta^{-1}}{2} \Delta u
\end{equation}
with the initial condition $u(x,0) = f(x)$. In the limit $\beta^{-1} \rightarrow 0$, (\ref{eqn1}) reduces to the non-viscous HJ equation
$$u_t = -\frac{1}{2} | \nabla u |^2,$$
whose viscosity solution is exactly the Moreau envelope \cite{Moreau_65}:
$$
u(x,t) = \inf_y \Big\{f(y) + \frac{1}{2t}\|y-x\|^2 \Big\}.
$$
The gradient descent dynamics for $f_{\gamma}$ is obtained by taking the limit of the following system of stochastic differential equation as the homogenization parameter $\varepsilon\to 0$:
\begin{align}  \label{eqn2}
	dx(s) &= -\gamma^{-1} (x-y) ds \nonumber \\
    dy(s) &= - \frac{1}{\varepsilon} \big[ \nabla f(y) + \frac{y-x}{\gamma} \big] ds + \frac{\beta^{-1/2}}{\sqrt{\varepsilon}} dW(s)
\end{align}
where $W(s)$ is the standard Wiener process. Specifically, we have 
$$
-\nabla f_{\gamma} (x) = - \gamma^{-1} ( x- \langle y \rangle)
$$ 
with $\langle y \rangle = \lim_{T \rightarrow \infty} \frac{1}{T} \int_0^T y(s) ds$ and $y(s)$ being the solution of (\ref{eqn2}) for fixed $x$. This gives rise to the implementation of Entropy-SGD \cite{deep_relax}:
\begin{align*}
y^{l+1, k} = & \,  y^{l, k} - \eta_y\Big (\nabla_l f(y^{l, k})+ \frac{y^{l,k} - x^k}{\gamma^k} \Big) + \sqrt{\eta_y\beta^{-1}} \varepsilon \quad (\mbox{inner loop})\\
x^{k+1} = & \, x^k - \eta_x\frac{x^k - \langle y \rangle^k}{\gamma^k} \quad (\mbox{outer loop})
\end{align*}
where $\eta_y$ and $\eta_x$ are the gradient step sizes for the inner and outer loops, respectively,  $\langle y \rangle^k$ is the moving average of $\{y^{l,k}\}$ output from the inner loop, and  $\sqrt{\eta_y\beta^{-1}} \varepsilon$ introduces the noise. Stochastic backward Euler simplifies Entropy-SGD in two aspects. First, the term $ \sqrt{\eta_y\beta^{-1}} \varepsilon$ is absent in SBE as the mini-batch gradient $\nabla_l f$ itself already contains the noise. Second, the step sizes $\eta_y$ and $\eta_x$ are both set to $\gamma^k$, which make the algorithm simpler with less tunable parameters.

\section{Experimental results}
We show by several experiments that the proposed stochastic backward Euler (SBE) gives superior clustering results compared with the state-of-the-art algorithms for $k$-means. SBE scales well for large problems. In practice, only a small number of fixed-point iterations are needed in the inner loop, and this seems not to depend on the size of the problem. Specifically, we chose the parameters \verb"imaxit" = 5 or $10$ and the averaging parameter $\alpha = 0.75$ in all experiments. We remark that SBE is not very sensitive to these parameters. For example, it works equally well for $\alpha = 0.9$. In addition, we always set $\gamma^0 = K$.

\subsection{2-D synthetic Gaussian data}
We generated 4000 synthetic data points in 2-D plane by multivariate normal distributions with 1000 points in each cluster. The means and covariance matrices used for Gaussian distributions are as follows:
\begin{align*}
& \mu_1 = \begin{bmatrix} -5 \\ -3 \end{bmatrix}, \;
\mu_2 = \begin{bmatrix} 5 \\-3 \end{bmatrix}, \;
\mu_3 = \begin{bmatrix} 0.0 \\ 5.0 \end{bmatrix}, \;
\mu_4 = \begin{bmatrix} 2.5 \\ 4.0 \end{bmatrix}; \\
& \Sigma_1 = \begin{bmatrix}0.8 & 0.1\\ 0.1 & 0.8 \end{bmatrix}, \; 
\Sigma_2 = \begin{bmatrix} 1.2 & 0.6\\ 0.6 & 0.7 \end{bmatrix}, \;
\Sigma_3 = \begin{bmatrix} 0.5 & 0.05\\ 0.05 & 1.6 \end{bmatrix}, \;
\Sigma_4 = \begin{bmatrix} 1.5 & 0.05\\ 0.05 & 0.6 \end{bmatrix}.
\end{align*}
For the initial centroids given below,both Lloyd's algorithm (or EM) and mini-batch EM got stuck at the same local minimum with objective value about $1.34$; see the left plot of Fig. \ref{fig1}. 
\begin{align*}
x_1 = \begin{bmatrix} -5.5989 \\ -2.7090 \end{bmatrix}, \;
x_2 = \begin{bmatrix} -4.4572 \\ -4.0614 \end{bmatrix}, \;
x_3 = \begin{bmatrix} -0.1082 \\ 5.2889 \end{bmatrix}, \;
x_4 = \begin{bmatrix} 2.3485 \\ 3.5286 \end{bmatrix}.
\end{align*}
Starting from where EM and mini-batch EM got stuck, we can see that SBE managed to jump over the trap of local minimum and arrived at a better minimum, which seems to be the global minimum; see the right plot of Fig. \ref{fig1}. 

\begin{figure}
\begin{tabular}{cc}
\includegraphics[width=0.45\textwidth]{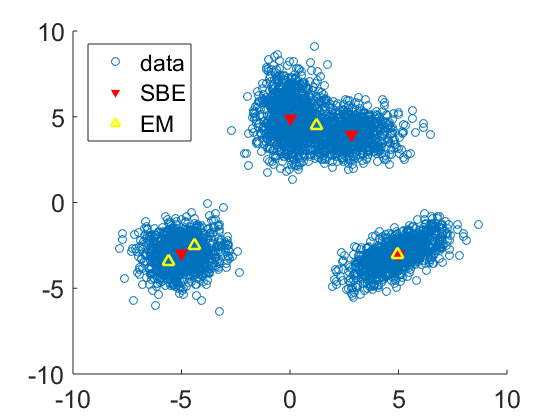}
\includegraphics[width=0.45\textwidth]{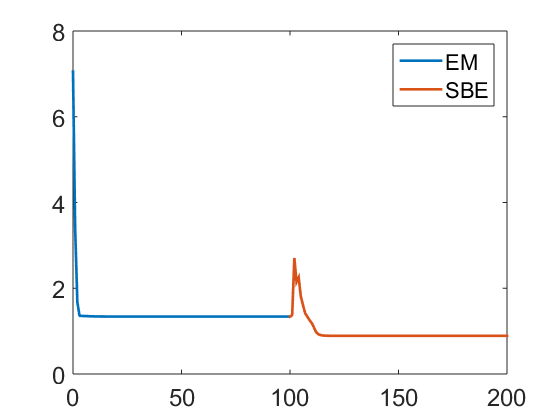}
\end{tabular}
\caption{Synthetic Gaussian data with 4 centroids. Left: Computed centroids by EM and SBE corresponding to the objective values 1.34 and 0.89, respectively. Right: Plot of objective value v.s. number of (outer) iteration. EM converged quickly but got trapped at a local minimum. SBE bypassed this local minimum and reached a better minimum by jumping over a hill.} 
\label{fig1}
\end{figure}

\subsection{Iris dataset}
The Iris dataset, which contains 150 4-D data samples from 3 clusters, was used for comparisons of SBE, EM as well as mini-batch EM algorithms. 100 runs were realized with the initial centroids randomly selected from the data samples. For the parameters, we chose mini-batch size $M=60$, initial step size, \verb"imaxit"$=40$, \verb"omaxit"$=10$, and decay parameter $\beta = \frac{1}{1.01}$. The histograms in Fig.~\ref{fig:iris} record the frequency of objective values given by the three algorithms. Clearly there was $29\%$ chance that EM got stuck at a local minimum whose value is about 0.48, whereas both SBE and mini-batch EM managed to locate an improved minimum valued at around 0.264 \emph{every time}. 

\begin{figure}
\centering
\begin{tabular}{cc}
\includegraphics[width=0.45\textwidth]{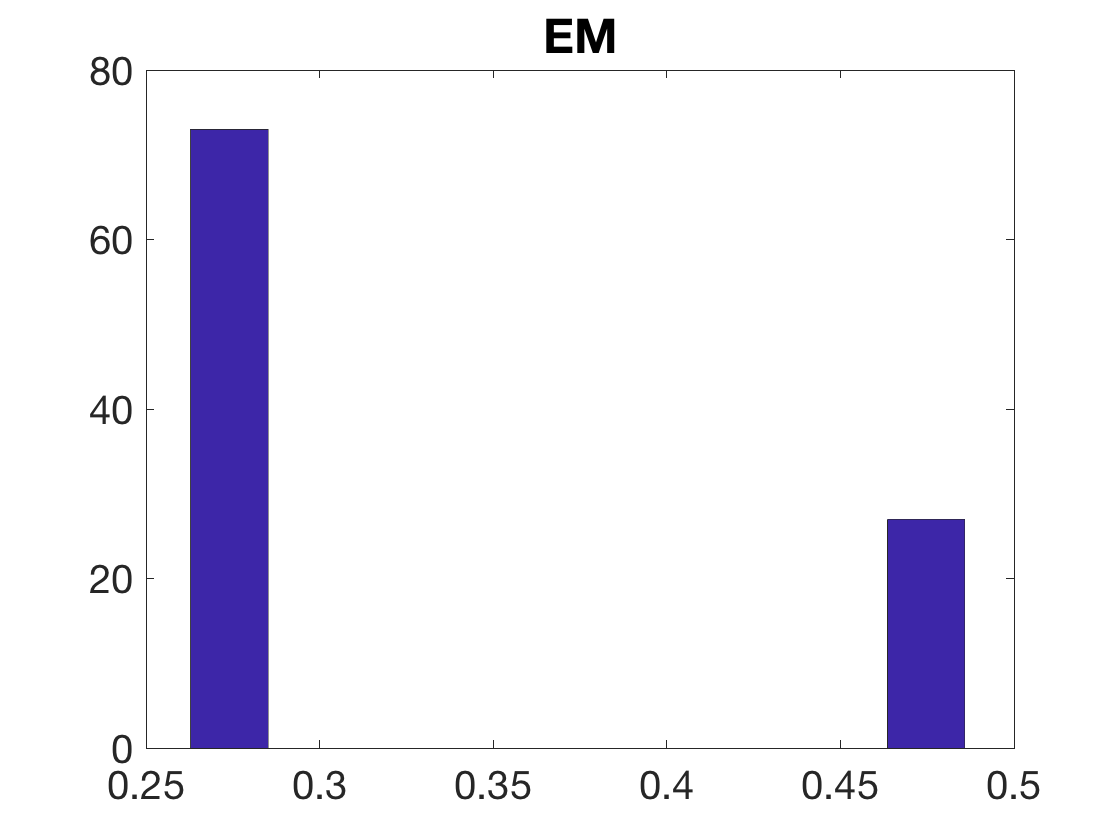} 
\includegraphics[width=0.45\textwidth]{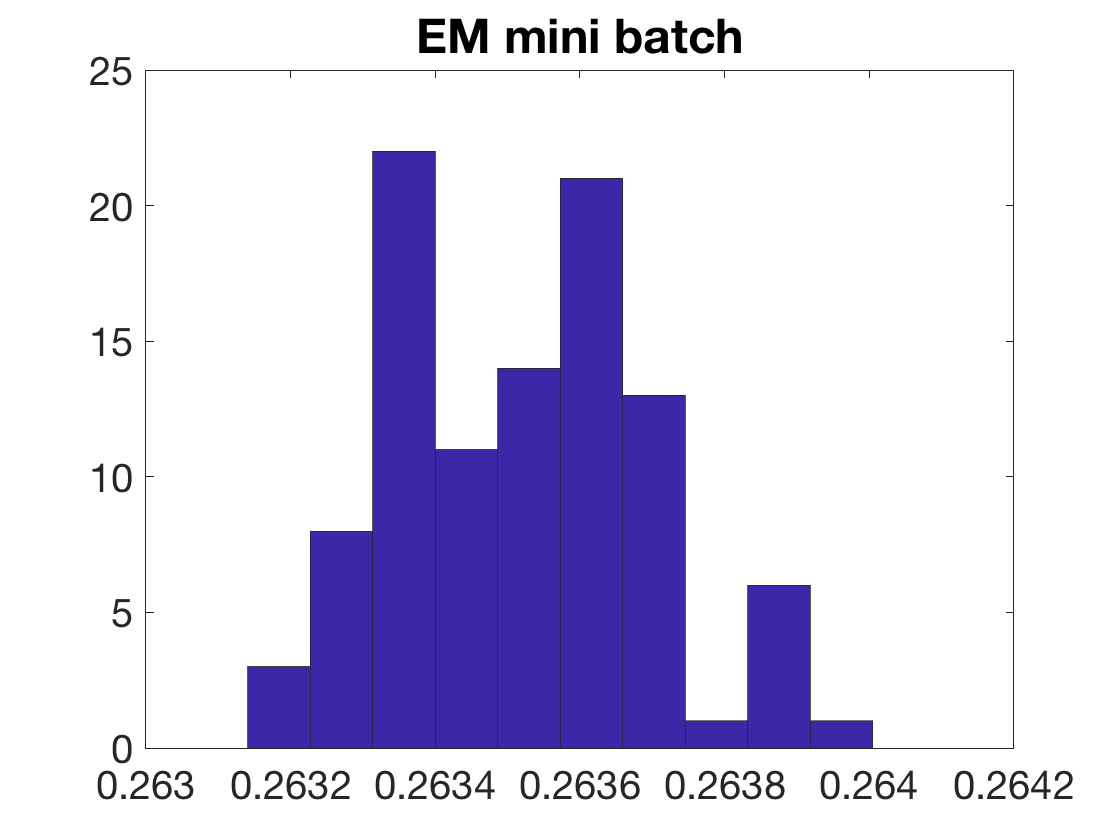} 
\\
\includegraphics[width=0.45\textwidth]{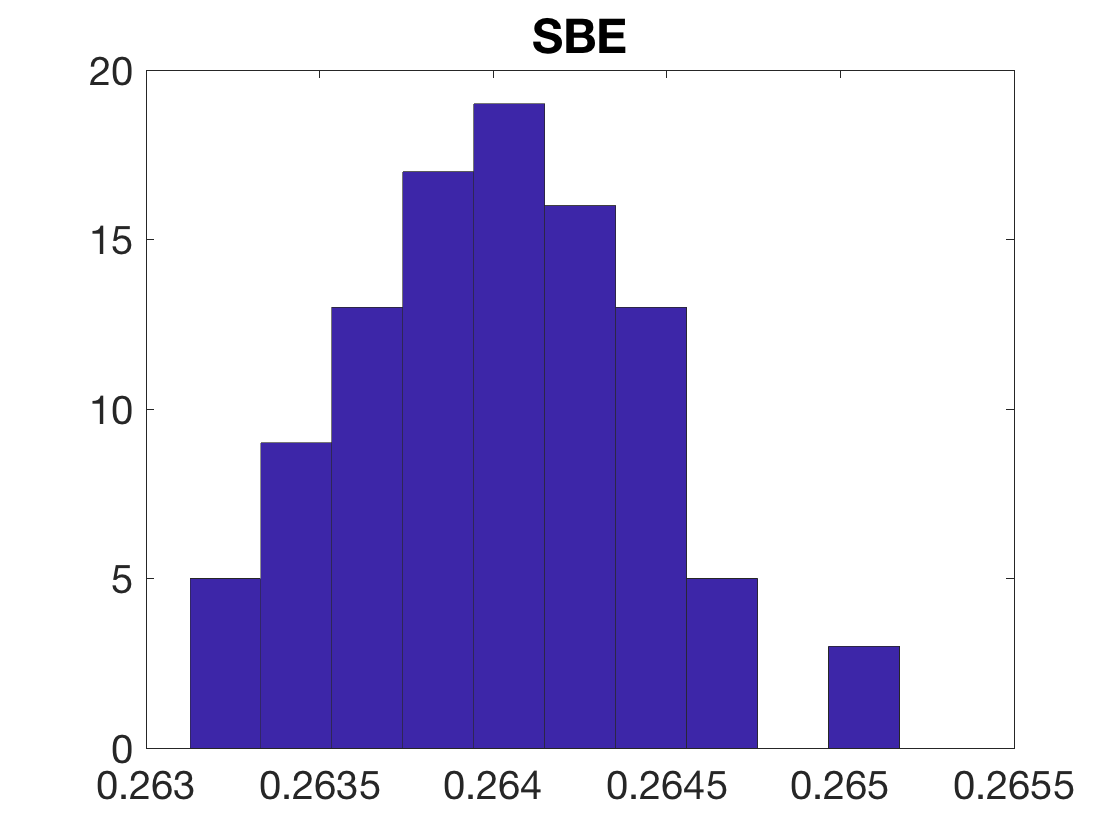} 
\includegraphics[width=0.45\textwidth]{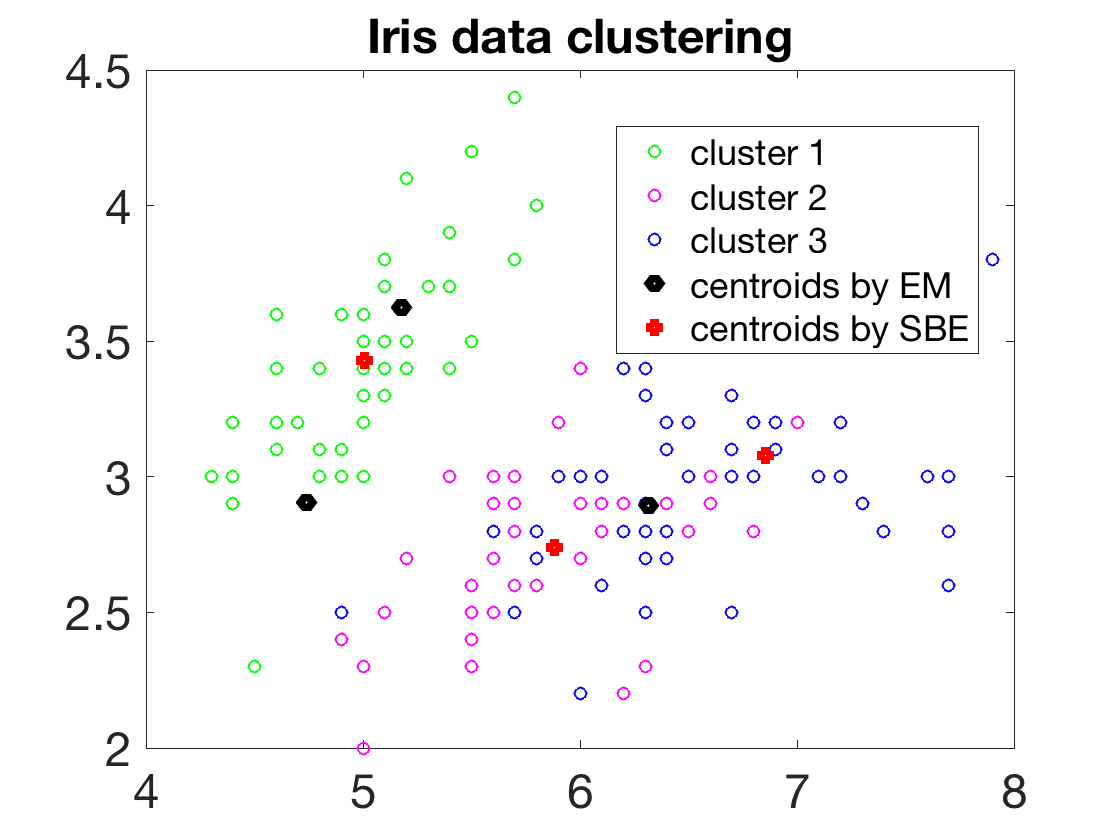}
\end{tabular}
\caption{The Iris dataset with 3 clusters. Top left: histogram of objective values obtained by EM in 100 trials. Top right: histogram of objective values obtained by mini-batch EM in 100 trials. Bottom left: histogram of objective values obtained by SBE (proposed) in 100 trials. Bottom right: computed centroids by EM (black) and SBE (red), corresponding to the objective values 0.48 and 0.264, respectively.}
\label{fig:iris}
\end{figure}

\subsection{Gaussian data with MNIST centroids}
We selected 8 hand-written digit images of dimension $28\times 28 = 784$ from MNIST dataset shown in Fig. \ref{fig2}, and then generated 60,000 images from these 8 centroids by adding Gaussian noise. We compare SBE with both EM and mini-batch EM (mb-EM) \cite{mnbatch,stochastic_kmeans} on 100 independent realizations with random initial guess. For each method, we recorded the minimum, maximum, mean and variance of the 100 objective values by the computed centroids.
\begin{figure}
    \centering
    \includegraphics[scale=0.6]{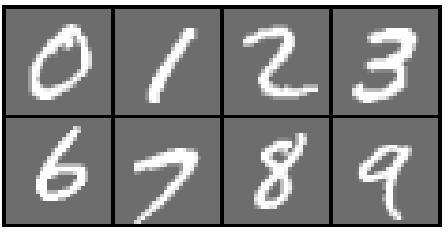}
    \caption{8 selected images from MNIST dataset. 60,000 sample images are generated from these 8 images by adding Gaussian noise.}
    \label{fig2}
\end{figure}

We first compare SBE and EM with the true number of clusters $K=8$. For SBE, mini-batch size $M = 1000$, maximum number of iterations for backward Euler \verb"omaxit"=150, maximum fixed-point iterations \verb"imaxit"= 10 for SBE. We set the maximum number of iterations for EM to be 50, which was sufficient for its convergence. The results are listed in the first two rows of Table~\ref{table1}. We observed SBE always found a minimum around 15.68 up to a tiny error due to the noise from mini-batch. Moreover, note that although we run more iterations (taking the inner loop into account) for SBE than for EM, SBE actually requires less distance evaluations and is computationally cheaper compared with EM because of the small mini-batch. More details will be discussed in section \ref{sec:comp}.

In the comparison between SBE and mb-EM, we reduced mini-batch size to $M=500$, \verb"omaxit"$=100$, \verb"imaxit"$=5$ and tested for $K=6, 8, 10$. Table \ref{table1} shows that with the same mini-batch size, SBE outperforms mb-EM in all three cases, in terms of both mean and variance of the objective values. 
 
\renewcommand{\arraystretch}{1.2} 

\begin{table}
\centering
\begin{tabular}{|c|c|c|c|c|c|c|c|}
\hline
$K$ & Method & Batch size & Max iter & Min &  Max & Mean & Variance  \\ 
\hline
\multirow{2}{*}{8} & EM & 60000 & 50 & 15.6800 & 27.2828 & 20.0203 & 6.0030  \\ 
\cline{2-8}
 & SBE  & 1000 & (150,10) & 15.6808 & 15.6808 & 15.6808 & 1.49$\times 10^{-10}$ \\ 
\hline

\multirow{2}{*}{6}  &  mb-EM & 500 & 100 & 20.44 & 23.4721 & 21.8393 & 0.67 \\ 
\cline{2-8}
& SBE   & 500  & (100,5)  & 20.2989 & 21.2047 & 20.4939 & 0.0439 \\ \hline

\multirow{2}{*}{8}  & mb-EM & 500 & 100 & 15.9193 & 18.5820 & 16.4009 & 0.7646 \\ 
\cline{2-8}
 & SBE   & 500 & (100,5) & 15.6816 & 15.6821 & 15.6820 & 1.18$\times 10^{-9}$ \\ 
\hline

\multirow{2}{*}{10} & mb-EM  & 500 & 100 & 15.9148 & 18.1848 & 16.1727 & 0.4332 \\ 
\cline{2-8}
& SBE  & 500& (100,5) & 15.6823 & 15.6825 & 15.6824 & 1.5$\times 10^{-9}$ \\ 
\hline

\end{tabular}\label{table1}
\caption{Gaussian data generated from MNIST centroids by adding noise. Ground truth $K=8$. Clustering results for 100 independent trails with random initialization.}
\end{table}

\subsection{Raw MNIST data}
In this example, We used the 60,000 images from the MNIST training set for clustering test, with 6000 samples for each digit (cluster) from 0 to 9. The comparison results are shown in Table~\ref{table2}. We conclude that SBE consistently performs better than EM and mb-EM. The histograms of objective value by the three algorithms in the case $K=10$ are plotted in Fig.~\ref{fig3}.

\begin{table}
\centering
\begin{tabular}{|c|c|c|c|c|c|c|c|}
\hline
$K$ & Method & Batch size & Max iter & Min &  Max & Mean & Variance  \\ 
\hline
\multirow{2}{*}{10} & EM & 60000 & 50 & 19.6069 & 19.8195 & 19.6725 & 0.0028  \\ 
\cline{2-8}
 & SBE  & 1000 & (150,10) & 19.6087 & 19.7279 & 19.6201 & 5.7$\times 10^{-4}$ \\ 
\hline
\multirow{2}{*}{8}  & mb-EM & 500 & 100 & 20.4948 & 20.7126 & 20.5958 & 0.0018 \\ 
\cline{2-8}
 & SBE   & 500 & (100,5) & 20.2723 & 20.4104 & 20.3090 & 0.0014 \\ 
\hline

\multirow{2}{*}{10} & mb-EM  & 500 & 100 & 19.9029 & 20.2347 & 20.0146 & 0.0041 \\ 
\cline{2-8}
& SBE  & 500 & (100,5) & 19.6103 & 19.7293 & 19.6354 & 0.0011 \\ 
\hline

\multirow{2}{*}{12}  &  mb-EM & 500 & 100 & 19.3978 & 19.7147 & 19.5136 & 0.0042 \\ 
\cline{2-8}
& SBE   & 500  & (100,5)  & 19.0492 & 19.1582 & 19.0972 & 6.2$\times 10^{-4}$ \\ \hline

\end{tabular}\label{table2}
\caption{Raw MNIST training data. The ground truth number of clusters is $K=10$. Clustering results for 100 independent trails with random initialization.}
\end{table}

\begin{figure}
\begin{tabular}{cc}
\includegraphics[width=0.45\textwidth]{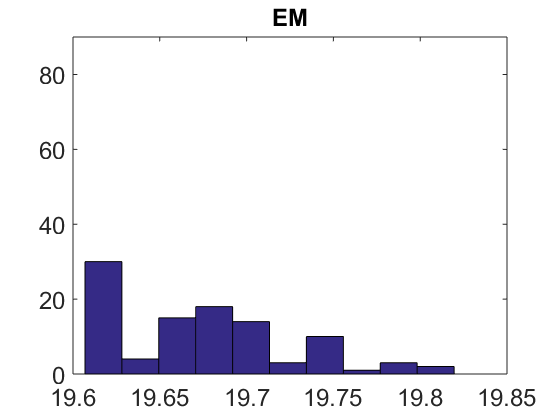}
\includegraphics[width=0.45\textwidth]{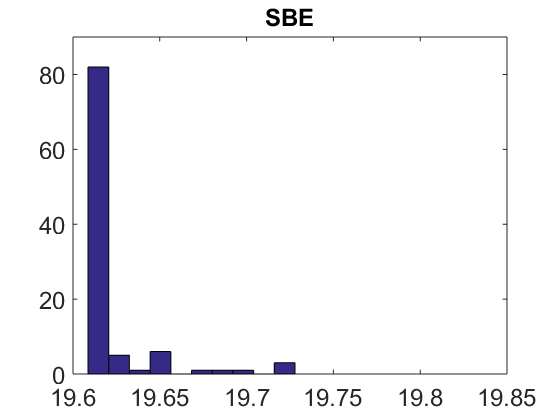}\\
\includegraphics[width=0.45\textwidth]{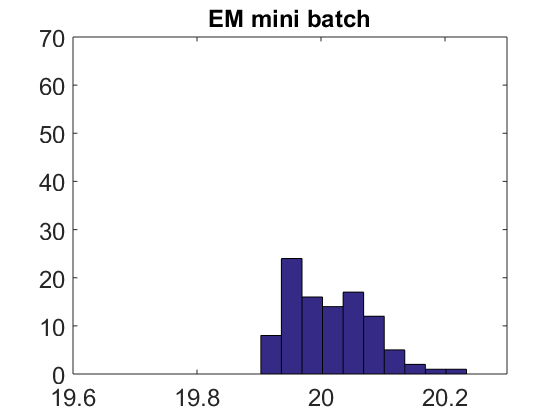}
\includegraphics[width=0.45\textwidth]{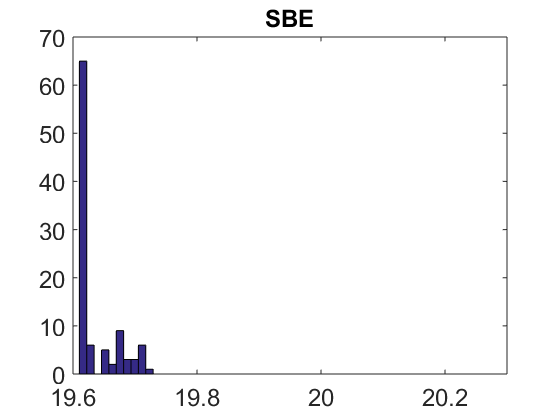}
\end{tabular}
    \caption{Histograms of objective value for MNIST training data with ground truth number of clusters $K=10$. Top left: EM.  Top right: SBE, mini-batch size of 1000. Bottom left: mn-EM, mini-batch size of 500. Bottom right: SBE, mini-batch size of 500. }\label{fig3}
\end{figure}

\subsection{MNSIT features}
We extracted the feature vectors of MNIST training data prior to the last layer of LeNet-5 \cite{mnist_98}. The feature vectors have dimension 64 and lie in a better manifold compared with the raw data. The results are shown in Table~\ref{table3} and Fig.~\ref{fig4} and \ref{fig5}.

\begin{table}
\centering
\begin{tabular}{|c|c|c|c|c|c|c|c|}
\hline
$K$ & Method & Batch size & Max iter& Min &  Max & Mean & Variance  \\ 
\hline
\multirow{2}{*}{10} & EM & 60000 & 50 &1.6238 & 3.0156 & 2.1406 & 0.0977  \\ 
\cline{2-8}
 & SBE  & 1000 & (150,10) & 1.6238 & 1.6239  &1.6239 & 2.7$\times 10^{-10}$ \\ 
\hline
\multirow{2}{*}{8}  & mb EM & 500 & 100 &2.3428 & 3.5972 & 2.7157 & 0.0666 \\ 
\cline{2-8}
 & SBE   & 500  &  (100,5) & 2.2833 & 2.4311 & 2.3274 & 0.0015 \\ 
\hline

\multirow{2}{*}{10} & mb EM  & 500 & 100 & 1.6504 & 2.6676 & 2.1391 & 0.0712 \\ 
\cline{2-8}
& SBE  & 500 & (100,5) & 1.6239 & 1.6242 & 1.6240 & 1.37$\times 10^{-9}$ \\ 
\hline

\multirow{2}{*}{12}  &  mb EM & 500 & 100 & 1.5815 & 2.6189 & 1.7853 & 0.0661 \\ 
\cline{2-8}
& SBE   & 500  & (100,5) & 1.5326 & 1.5891 & 1.5622 & 9.8$\times 10^{-5}$ \\ \hline

\end{tabular}
\caption{MNIST features generated by LeNet-5 network. The ground truth number of clusters is $K=10$. Clustering results for 100 independent trails with random initialization.}\label{table3}
\end{table}

\begin{figure}
\begin{tabular}{cc}
    \includegraphics[width=0.45\textwidth]{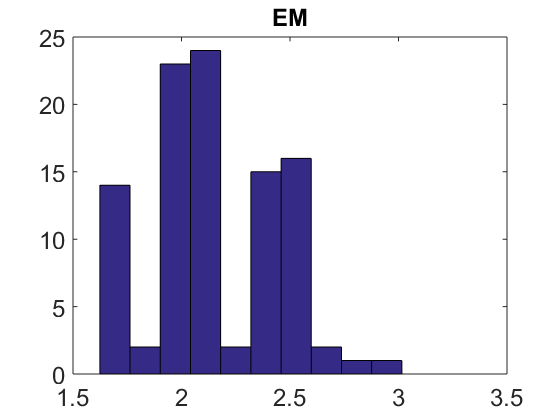}
    \includegraphics[width=0.45\textwidth]{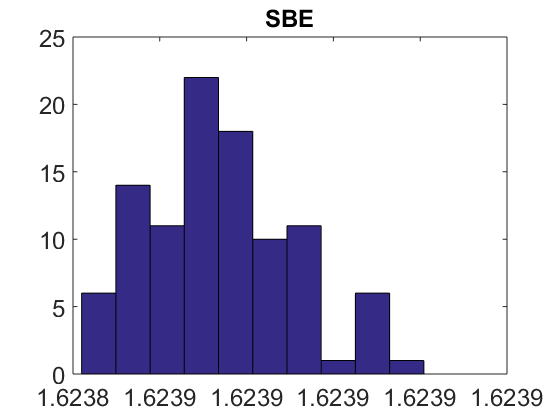}\\
     \includegraphics[width=0.45\textwidth]{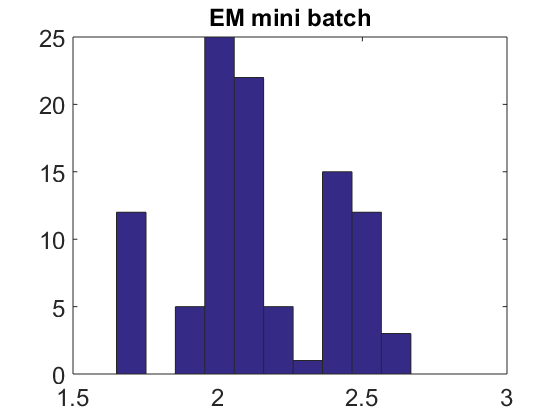}
 \includegraphics[width=0.45\textwidth]{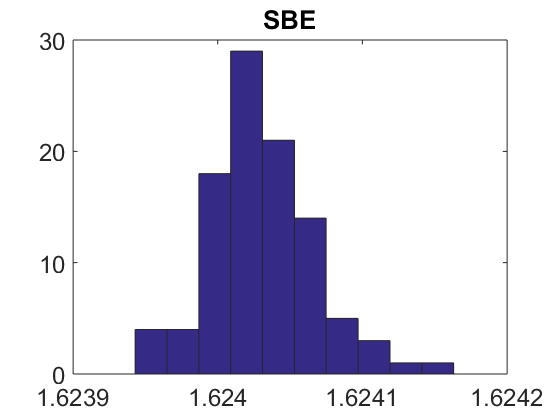}
\end{tabular}
    \caption{Histograms of objective value for MNIST feature data with ground truth number of clusters K=10. Top left: EM.  Top right: SBE, mini-batch size of 1000. Bottom left: mn-EM, mini-batch size of 500. Bottom right: SBE, mini-batch size of 500.}\label{fig4}
\end{figure}

\begin{figure}
    \centering
    \includegraphics[width=0.5\textwidth]{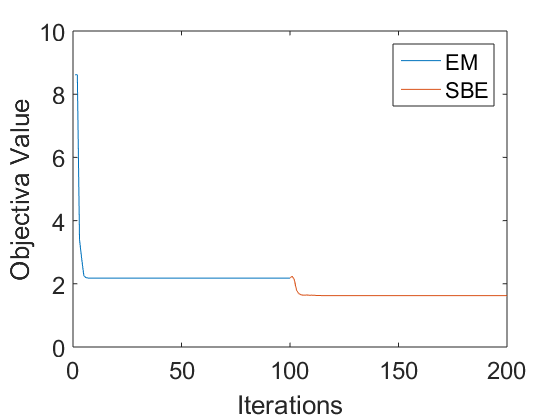}
    \caption{Objective value for MNIST features dataset. The ground truth number of clusters is $K=10$. EM got trapped at local minimum around 2.178. Initializing SBE with this local minimizer, an improved minimum around 1.623 was found.}\label{fig5}
\end{figure}

\subsection{Comparison of time efficiency}\label{sec:comp}
We first compare the per-(outer)iteration costs for EM, mini-batch EM, and SBE, respectively. In every iteration, we need to find the the labels or clusters associated with data points that are used to update the centroids, for which minimum distance between the data points and current centroids are computed. This dominates the total computational cost, especially for big data. In EM, we need to find the labels for all data points when updating the centroids. In contrast, only a small batch of labels are needed in mini-batch EM and SBE. Specifically, for the datasets in sections 3.3 and 3.4 with 60,000 points of dimension 784, the per-iteration computation time of EM was around 2.3 seconds, whereas those of mini-batch EM and SBE (with 5 inner iterations) were 0.03 seconds and 0.1 seconds, respectively. For the raw MNIST data with $K = 10$, EM usually needed around 15 iterations to converge to a good minimum at about 19.6 with random initialization (if succeeded; see Table 2), whereas SBE needed 30 iterations and mini-batch EM always failed to do so. So basically we saw more than 10$\times$ savings in time efficiency of SBE, compared with EM. The tests were carried out on a laptop with  2.8 GHz Intel Core i7 CPU and 16 GB memory.

\section{Discussions}
In this section, we provide an intuitive explanation for why SBE often succeeds to find better local minima than SGD through a simple one-dimensional example in Fig. \ref{fig:sbe}. At the $(k+1)$-th iteration, SBE approximately solves $x = x^k - \gamma \nabla \phi(x)$ due to the noise introduced by mini-batch gradient. Since $\nabla \phi$ is technically only piecewise Lipschitz continuous, the backward Euler may have multiple solutions. For example, in Fig. \ref{fig:sbe}, we get two solutions $x_{\mathrm{BE},1}$ in the leftmost valley and $x_{\mathrm{BE},2}$ in the second from the left. $x^{k+1}$ solved by SBE is close to $x_{\mathrm{BE},2}$ as it gives a lower objective value of $f_\gamma$. Then $x^{k+1}$ bypasses the local minimum, which explains the increase of objective value shown in the right plot of Fig.~\ref{fig1}. We conjecture that averaging of the iterates $\{y^l\}$ enables an aggressive step size $\gamma$ larger than the theoretical upper-bound $1/L$ as in Theorem \ref{thm:fp}, which also helps skip bad local minima. We did observe blowup phenomenon numerically whenever the averaging scheme was not used. It is of our interest to prove an improved upper-bound for $\gamma$ in the future work. Similar to what was done in \cite{artina_13} , another direction is to analyze how exact the inner problems have to be solved in order to still guarantee the convergence of the outer Backward Euler problem.

\begin{figure}
\begin{tabular}{cc}
    \includegraphics[width=0.45\textwidth]{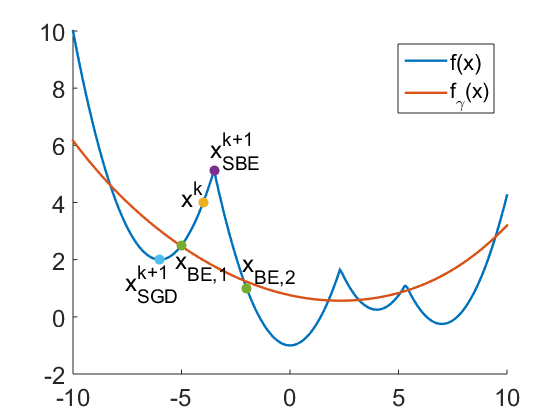}
    \includegraphics[width=0.45\textwidth]{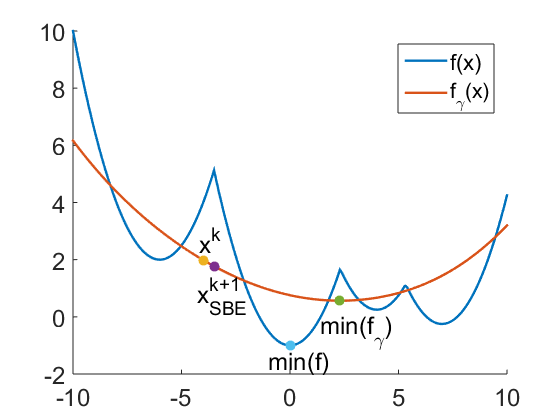}
\end{tabular}
 	\caption{Comparison the updates between SGD and SBE. Left: In the $(k+1)$-th update, SGD gives $\mathbb{E}_{\mathrm{SGD}}[x^{k+1}] = x^k - \gamma \nabla f(x^k)$ while SBE solves gradient descent on $f_\gamma$ and ends up with $x^{k+1}_{\mathrm{SBE}}$. Right: SBE tends to converge to the global minimum of local entropy $f_{\gamma}$. We must let $\gamma \rightarrow 0$ in order for SBE to converge to the true global minimum of $f$.}\label{fig:sbe} 
\end{figure}


\section*{Acknowledgement}
This work was partially supported by AFOSR grant FA9550-15-1-0073 and ONR grant N00014-16-1-2157. We would like to thank Dr. Bao Wang for helpful discussions. We also thank the anonymous reviewers for their constructive comments.


\end{document}